\documentclass[11pt]{amsart}

\usepackage{amsmath,amsfonts,amsthm,amssymb}
\usepackage{bbm}
\usepackage{graphicx}
\usepackage{fullpage,nicefrac,color,subfig}
\allowdisplaybreaks

\newtheorem{thm}{Theorem}
\newtheorem*{sch}{Schoenberg Embedding Theorem}

\newtheorem{lem}{Lemma}
\newtheorem{obs}{Observation}

\newtheorem{prop}{Proposition}
\newtheorem{claim}{Claim}

\newcommand{\mc}{\mathcal}
\newcommand{\R}{\mathbb{R}}

\newcommand{\N}{\mathbb{N}}

\DeclareMathOperator{\vol}{Vol}

\DeclareMathOperator{\SPAN}{span}

\newcommand{\one}{\mathbbm{1}}
\newcommand{\paren}[1]{\left(#1\right)}
\newcommand{\set}[1]{\left\{ #1 \right\}}
\newcommand{\inner}[3][]{\ensuremath{\left\langle #2, #3 \right\rangle_{#1}}}
\newcommand{\abs}[1]{\left| #1 \right|}

\newcommand{\half}{\nicefrac{1}{2}}
\newcommand{\norm}[2][]{\left\| #2 \right\|_{#1}}
\newcommand{\size}[1]{\left| #1 \right|}
\newcommand{\hidden}[1]{}

\begin{document}
\title{On the structure of Isometrically Embeddable Metric Spaces}
\author{Kathleen Nowak}
\address{Pacific Northwest National Laboratory, Richland, WA 99352}
\email{kathleen.nowak@pnnl.gov}
\author{Carlos Ortiz Marrero}
\address{Pacific Northwest National Laboratory, Richland, WA 99352}
\email{carlos.ortizmarrero@pnnl.gov}
\author{Stephen J. Young}
\address{Pacific Northwest National Laboratory, Richland, WA 99352}
\email{stephen.young@pnnl.gov}
\thanks{\textit{PNNL Information Release:} PNNL-SA-126943}

\begin{abstract}
Since its popularization in the 1970s the Fiedler vector of a graph has
become a standard tool for clustering of the vertices of the graph.
Recently, Mendel and Noar, Dumitriu and Radcliffe, and Radcliffe and
Williamson have introduced geometric generalizations of the Fiedler
vector.   Motivated by questions stemming from their work we provide
structural characterizations for when a finite metric space can be
isometrically embedded in a Hilbert space.
\end{abstract}

\maketitle

\section{Introduction}
Since the work of Fiedler~\cite{Fiedler:AlgConnect,Fiedler:AlgConnectGraphs,Fiedler:LaplacianConnect}, the spectrum of the combinatorial
Laplacian $L = T-A$, where $A$ is the adjacency matrix of a connected,
simple graph $G = (V,E)$ and $T$ is the diagonal matrix of degrees, has received
significant attention as a means of partitioning the vertices of the
graph, see for instance \cite{Hagen:SpectralCut,Pothen:FiedlerPartition}. Specifically, the smallest non-zero eigenvalue and its
associated eigenvector (typically referred to as  the algebraic
connectivity, $a(G)$, and the Fiedler vector) have been used to create
bipartitions with guaranteed quality metrics~\cite{Alon:LaplacianDiameter,Cheeger:Inequality,Mohar:Isoperimetric,Spielman:SpectralWorksConf,Spielman:SpectralWorks,Urschel:MaximalSpectralError}.  More recently,
the smallest non-zero eigenvalue $\lambda_2(G)$ of the normalized Laplacian,
$\mathcal{L} = T^{-\half}LT^{-\half} = I - T^{-\half}AT^{-\half}$, and
its associated eigenvector (which is sometimes also referred to as the
Fiedler vector) have also been used to perform Fiedler-like
clustering~\cite{Ng:SpectralClustering,Shi:NormalizedCuts}.  It is worth mentioning that both $a(G)$
and $\lambda_2(G)$ have numerous connections to structural properties
of graphs beyond graph partitioning, see for instance
\cite{Biggs:AGT,Brouwer:SpectraGraphs,Chung:spectral,Godsil:AGT,Mohar:Laplacian}.

Over the past decade or so, there have been several works
generalizing the ideas and quality bounds of the Fiedler partitioning schemes to
multi-way partitions~\cite{Alpert:MultiWaySpectral,Kwok:HighOrderCheeger,Lee:HighOrderCheeger,Lee:HighOrderCheegerConf,Ng:SpectralClustering}.  Generally, the idea behind these
approaches has been to use the eigenvectors associated with the $k$
smallest non-zero eigenvalues to define a mapping from the vertices
into $\R^k$.  At this point an approximation of a geometric clustering algorithm, such as
$k$-means, is used to form the multi-way partitions.  Recently an
alternative generalization of the Fiedler vector has been developed by Mendel
and Naor for the combinatorial Laplacian~\cite{Mendel:HadamardExpanderConf,Mendel:NonlinearSpectralCalc,Mendel:HadamardExpander} and subsequently by
Radcliffe and Williamson \cite{Radcliffe:GeometricEigs}, and Dumitriu and Radcliffe~\cite{Radcliffe:GeometricExpansion}
for the normalized Laplacian.  These generalizations are rooted
in the observation that, via the Courant-Fischer theorem, there is some
functional form $\mathcal{R}$ such that $a(G)$ and
$\lambda_2(G)$ can be expressed as $\min_{f \colon V \rightarrow \R} \mathcal{R}(f)$.  Furthermore, this functional form can
be manipulated so that it depends only on $\paren{f(u) - f(v)}^2$,
which is the square of the standard distance metric on $\R$.  Thus,
one can naturally defined the \emph{geometric Fiedler vector} by
replacing $\R$ with an arbitrary metric space $(X,d)$ and considering
$\lambda(G,X) = \min_{f \colon V \rightarrow X} \mathcal{R}(f)$.  

This definition leads naturally to the question of when can the
geometric Fiedler vector be thought of as an eigenvector of some
linear operator.  Specifically, which finite metric spaces $(X,d)$ are
consistent with a definition of orthogonality, i.e. can be
isometrically embedded in a Hilbert space. Although this question can
be answered from an analytical point of view by the Schoenberg
Embedding Theorem~\cite{Schoenberg:Embedding,Schoenberg:MetricPosDef},
we instead approach the question from a more structural viewpoint. Specifically, for every metric $(X,d)$ associate the 
\emph{critical graph} $G = (X,E)$ with weights given by $d$ and edge
set given by $\set{ \set{u,v} \in \binom{X}{2} \colon \forall z \neq
  u,v \  d(u,v) < d(u,z) + d(z,v) }$.  In this work we will provide a
complete characterization of the unweighted critical graphs
corresponding to metrics which are isometrically embeddable in a
Hilbert space.  Additionally, we provide structural characteristics
of weighted critical graphs that are isometrically embeddable in
a Hilbert space.

\section{Schoenberg's Spectral Characterization} \label{shoenberg}
Although we are taking a structural point of view  by identifying
critically graphs corresponding to isometrically embeddable metric
spaces, the analytic embedding theorem of Schoenberg will play a key
role in our analysis~\cite{Schoenberg:Embedding,Schoenberg:MetricPosDef}.  Before stating
Schoenberg's result we need to define what it means for a bi-variate
function on a metric space to be conditionally negative
  definite.  To that end, let 
$(X,d)$ be a metric space. A function $\psi: X \times X \rightarrow
\R$ is called \emph{conditionally negative definite} if for any $x_1, x_2,
\dots x_n \in X$ and $\alpha_1, \alpha_2, \dots, \alpha_n \in \R$ such
that $\sum_{j = 1}^n \alpha_j = 0$, we have 
\[
\sum_{i,j = 1}^n\alpha_i\alpha_j \psi(x_i,x_j) \leq 0. 
\] 
With terminology in hand we can now state the following theorem:
\begin{sch} \label{embed_thm}
  Let $(X,d)$ be a metric space. There exists a Hilbert space $\mc{H}$ and an isometry $\phi: X \rightarrow \mc{H}$ if and only if $d^2$ is conditionally negative definite. 
\end{sch}

As we are dealing with finite metric spaces, it will be
helpful to rephrase Schoenberg's result in terms of
matrices.  Specifically, given a finite metric space $(X,d)$ define
$D$ to be the squared distance matrix for the metric, that is, $D_{xy}
= d(x,y)^2$.  From Schoenberg's result we have that $(X,d)$ is
isometrically embeddable into a Hilbert space if and only if 
\[
\max_{\substack{\alpha \in \R^n \\ \alpha \perp \one}} \alpha^T D \alpha \leq 0,
\]
where $\one$ is the all ones vector.  
Further, note that, 
\[
\max_{\substack{\alpha \in \R^n \\ \alpha \perp \one}} \alpha^T D \alpha  = \max_{v \in \R^n}v^T\left(I-\frac{1}{n}J\right)^TD\left(I-\frac{1}{n}J\right)v,
\]
where $J = \one^T\one$ is the matrix of all ones. 
Hence, $(X,d)$ is isometrically embeddable if and only if $\left(I-\frac{1}{n}J\right)^TD\left(I-\frac{1}{n}J\right)$ is negative semi-definite. 

Although not necessary for our study of critical graphs, it is
interesting to note that the standard proofs of
Schoenberg's result for finite metric spaces, see for instance \cite{Paulsen:RPKHilbert}, are
constructive.  Specifically, given an arbitrary fixed point $x_0 \in
X$, the Hilbert space is given by $\mathcal{H}
= \set{ \sum_{x \in X} \alpha_x k_x \colon \alpha_x \in \R}$ with $\inner{k_x}{k_y}
= d(x,x_0)^2 + d(y,x_0)^2 - d(x,y)^2$ and the isometry is given by
$\phi(x) = \frac{1}{\sqrt{2}}k_x$ for all $x$.    In particular, the
Hilbert space is isomorphic to the Hilbert space on $\R^{| X |}$
with the inner product given by $\alpha^T K \beta$ for $\alpha,\beta
\in \R^{|X|}$ and $K_{x,y} =  d(x,x_0)^2 + d(y,x_0)^2 - d(x,y)^2$.

Since the kernel function $K$ completely determines the Hilbert space
it is worth addressing how the choice of $x_0$ affects $K$.  To that
end consider the metric space formed by shortest path distances on
$P_3$, the path with three vertices.  When considering the three $K$
matrices generated by the vertices of $P_3$ we get 
\[
\left[ \begin{matrix}
0 & 0 & 0  \\
0 & 2 & 4 \\
0 & 4 & 8
\end{matrix}\right],
\quad 
\left[ \begin{matrix}
2 & 0 & -2  \\
0 & 0 & 0 \\
-2 & 0 & 2 
\end{matrix}\right],
\textrm{and\quad}
\left[ \begin{matrix}
8 & 4 & 0  \\
4 & 2 & 0 \\
0 & 0 & 0
\end{matrix}\right].
\]
The left and right kernels are clearly similar, reflecting the unique
non-trivial automorphism of $P_3$, but one can easily see that the
central kernel is not similar to the left or the right kernel by
noting that the trace is 4 as compared with 10.  To our knowledge it
is unknown if given generating points $x,x' \in X$ resulting in Hilbert
spaces $\mathcal{H}$, $\mathcal{H'}$ and isometries $\phi, \phi'$
there is an isometry $\psi$ between
$\mathcal{H}$ and $\mathcal{H'}$ such that $\phi' = \psi \circ \phi$.

\section{Prior Work}
Besides the work of Schoenberg there is an extensive literature
devoted to understanding the nature of embedding metric spaces in
other spaces, typically Hilbert spaces.  The work that is perhaps
closest in spirit to this work is recent work towards determining the
structure of (strictly) conditionally negative definite
matrices~\cite{Hjorth:StrictNegMetric,Joziak:NegDefKernels,Winkler:GraphsNegType}.  In
particular, the work
of Joziak and Winkler~\cite{Joziak:NegDefKernels,Winkler:GraphsNegType}
is very similar to this work in spirit, in that they concern
themselves with the graphical structures that result in conditionally
negative definite matrices.  Specifically, they consider when the
matrix for the shortest path metric in a graph is conditionally
negative definite.  In contrast with these works, we will be focussed on when
the entrywise square of the distance matrix is conditionally negative
definite.  Somewhat surprisingly, this results in significantly
different characterizations. Specifically, Joziak shows that a large
class of graphs including trees and odd cycles are conditionally
negative definite and we show that almost all of these metrics are not
isometrically embeddable in a Hilbert space with the only exceptions
being the path and the triangle. 

If, instead of requiring an isometric embedding into a Hilbert space, we allow for the
distances between elements to contract we end up considering what are
termed \emph{low-distortion embeddings}.  A survey of this broad
ranging field is beyond the scope of this work but an introduction
by Indyk and Matou\c{s}ek can be found in~\cite{Handbook:DiscGeom}.
At this point we would be remiss if we did nto also mention the connection between the
geometric Fiedler vector for the normalized Laplacian over the metric
generated by $K_2$.  Specifically, in this case the geometric Fiedler
vector gives the solution to the \textsc{Sparsest-Cut} problem with
pairwise unit demands.  This problem is known to be
\textsc{NP-complete} and thus can only be approximately solved
efficiently.  Several of these approximation schemes involve
calculating an auxiliary metric via linear programming and then
embedding this metric in Euclidean space with small distortion.  In
fact, the guarantees on the distortion often directly lead to the approximation
guarantees for the algorithm.  See for instance \cite{Arora:ExpanderFlows,Arora:EuclidDistort,Chawla:NegTypeEmbed,Karuthgamer:MeasuredDescent,Lee:RelaxationCutCone}.

Finally, we wish to mention the work of Graham and Winkler who
consider the question of whether a graph $G$ can be isometrically
embedded in a graph $G^*$ which is the Cartesian product of simpler
graphs.  Although there is no universal graph $H$ such that every
graph can be isometrically embedded in $H^m$ for some $m$, Graham and
Winkler provide, given a graph $G$, a canonical means of constructing
factors and a product graph $G^*$ such that there exists an isometric
embedding of $G$ in to $G^*$~\cite{Graham:IsometricEmbedConf,Graham:IsometricPNAS,Graham:IsometricEmbed,Graham:IsometricEmbedCorr}.

\section{Structural Characterizations of Isometric Embeddings}
There is a natural correspondence between finite metric spaces and
connected weighted graphs. For a connected weighted graph $G = (V, E,
w)$, the shortest path distance is a metric on $V$. Conversely, given
a metric space $(X,d)$, we say that a weighted graph $G = (V,E,w)$
\emph{represents} or \emph{generates} the metric if there exists a bijection $\phi: X \rightarrow V$
such that for all $x,x' \in X$, $d_G(\phi(x),\phi(x')) =
d_X(x,x')$. For example, $(X, \binom{X}{2}, d)$ generates the metric
space $(X,d)$. However, this graph representation is not unique. One
can often remove edges from this graph without altering the metric. We
will show, however, that there is a unique minimal graph representing the
metric and that this graph is the critical graph for the metric.
\begin{prop}
For any finite metric space the associated critical graph is the
unique graph with the minimum number of edges that generates the metric.
\end{prop}
\begin{proof}
Let $(X,d)$ be a finite metric space and let $G = (X,E,w)$ be the
critical graph that is associated to the metric.   First we will
show that $G$ generates the metric $(X,d)$. To that end, let
$d^G(\cdot,\cdot)$ be the metric on $X$ induced by $G$. Since $d^G$ is
the shortest path metric on $G$, we have that 
for any $u \neq v$, there exists a path $u = x_0, x_1,
\ldots, x_{t+1} = v$ in $G$ such that 
\[ d^G(u,v) = \sum_{i=0}^t w(x_i,x_{i+1}) = \sum_{i=0}^t
  d(x_i,x_{i+1}) \geq d(u,v).\]
Thus $d^G$ is an upper bound on $d$.  
Suppose now that $d \neq d^G$, then the set $\set{ \set{u,v} \colon
  d(u,v) < d^G(u,v)}$ is non-empty and hence has an element $(x,y)$
which minimizes $d(x,y)$.  Now since $d(x,y) < d^G(x,y)$ we know that
$x\not\sim y$ in $G$, and thus there is some other element $z \in X$ such that
$d(x,y) = d(x,z) + d(z,y)$.  But by the construction of $\set{x,y}$,
$d(x,z) = d^G(x,z)$ and $d(z,y) = d^G(z,y)$, and hence $d(x,y) <
d^G(x,y) \leq d^G(x,z) + d^G(z,y) = d(x,z) + d(z,y) = d(x,y)$, a
contradiction.  

To complete the proof it suffices to show that every edge in the
critical graph $G$ must
be present in every graph that generates the metric space $(X,d)$.
However, this is obvious as if $\set{x,y} \in E$ then $d(x,y) < d(x,z)
+ d(z,y)$ for every $z$ other than $x$ or $y$, and thus the metric can
not be recovered for $\set{x,y}$ via any non-trivial path.
\end{proof}

From the discussion above, it suffices to view a finite metric space as a weighted graph adorned with the shortest path distance. Let $H$ be a weighted graph metric with squared distance matrix $D$. Then, as remarked above, $H$ can be isometrically embedded into a Hilbert space if and only if 
\[
\min_{\substack{\alpha \in \R^n \\ \alpha \perp \mathbbm{1}}} \alpha^T D \alpha \leq 0.
\]
Next we characterize which weighted graphs satisfy this criterion. 
\subsection{Unweighted Graphs}
 We first consider the case where the weights in the critical graph are
all the same, that is, where the critical graph may be viewed as an
unweighted graph.

\begin{thm}
Let $H = (V,E,w)$ be a weighted graph with weight function $w = \mathbbm{1}$. Then $H$ can be isometrically embedded in a Hilbert space if and only if $H$ is a path or a complete graph. 
\end{thm}

\begin{proof}
We start by analyzing the subgraph induced by the closed neighborhood of a vertex. 
\begin{claim}
If $\min\limits_{\substack{\alpha \in \R^n \\ \alpha \perp \mathbbm{1}}} \alpha^T D \alpha \leq 0$, then for any vertex $v$ with $d_v > 2$, $H[N[v]]$ must be complete. 
\end{claim}

\begin{proof}
Suppose that $H$ contains a vertex $v$ of degree at least three such that $H[N[v]] \neq K_{d_v+1}$. Let $u$ and $w$ be non-adjacent neighbors of $v$ and let $z$ be a third neighbor. Note that the graph $H'$ induced by $\{v,u,z,w\}$ is a shortest path subgraph of $H$.  Further, $H'$ must be one of the graphs shown in Figure \ref{configs} with corresponding submatrix of $D$.
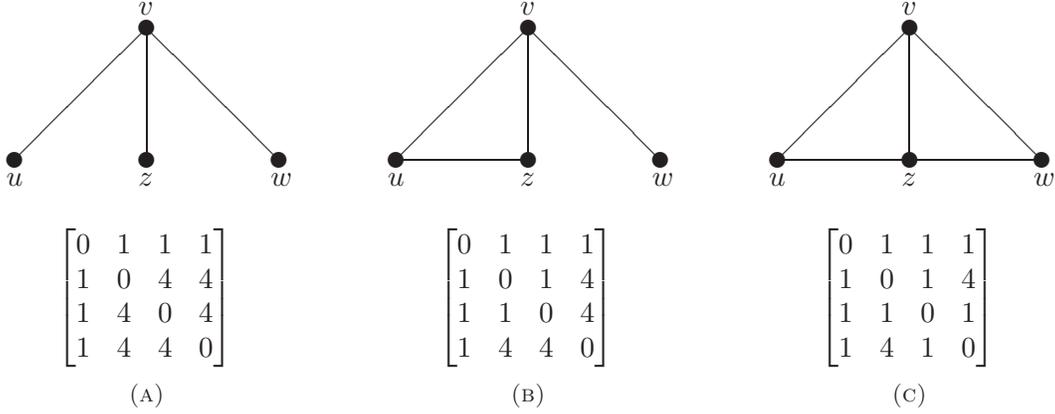
\begin{figure}
\centering
\hfill
\subfloat[]{
\begin{picture}(100,140)
\put(0,80){\circle*{6}}
\put(50,80){\circle*{6}}
\put(100,80){\circle*{6}}
\put(50,130){\circle*{6}}
\put(50,130){\line(-1,-1){50}}
\put(50,130){\line(0,-1){50}}
\put(50,130){\line(1,-1){50}}
\put(-3,70){$u$}
\put(47,70){$z$}
\put(97,70){$w$}
\put(47,135){$v$}
\put(16,25){$\left[ \begin{matrix} 0 &  1 & 1 & 1 \\ 1 & 0 & 4 & 4 \\ 
      1 & 4 & 0 & 4 \\ 1 & 4 & 4 & 0 \end{matrix}\right]$}
\end{picture}
}
\hfill
\subfloat[]{
\begin{picture}(100,140)
\put(0,80){\circle*{6}}
\put(50,80){\circle*{6}}
\put(100,80){\circle*{6}}
\put(50,130){\circle*{6}}
\put(50,130){\line(-1,-1){50}}
\put(50,130){\line(0,-1){50}}
\put(50,130){\line(1,-1){50}}
\put(0,80){\line(1,0){50}}
\put(-3,70){$u$}
\put(47,70){$z$}
\put(97,70){$w$}
\put(47,135){$v$}
\put(16,25){$\left[ \begin{matrix} 0 &  1 & 1 & 1 \\ 1 & 0 & 1 & 4 \\ 
      1 & 1 & 0 & 4 \\ 1 & 4 & 4 & 0 \end{matrix}\right]$}
\end{picture}
}
\hfill
\subfloat[]{
\begin{picture}(100,140)
\put(0,80){\circle*{6}}
\put(50,80){\circle*{6}}
\put(100,80){\circle*{6}}
\put(50,130){\circle*{6}}
\put(50,130){\line(-1,-1){50}}
\put(50,130){\line(0,-1){50}}
\put(50,130){\line(1,-1){50}}
\put(0,80){\line(1,0){100}}
\put(-3,70){$u$}
\put(47,70){$z$}
\put(97,70){$w$}
\put(47,135){$v$}
\put(16,25){$\left[ \begin{matrix} 0 &  1 & 1 & 1 \\ 1 & 0 & 1 & 4 \\ 
      1 & 1 & 0 & 1 \\ 1 & 4 & 1 & 0 \end{matrix}\right]$}
\end{picture}
}
\hfill
\hfill
\caption{Possible Neighborhood Configurations and Associated Principle
Submatrices}\label{configs}
\end{figure}
In each case, there exists a vector $\alpha$ such that $\alpha^TD\alpha >0$. For example, if the vertices are ordered starting with $v, u, w, z$, then assigning $\alpha$ to be $ (-3, 1, 1, 1, 0, \dots, 0)$, $(4, -1, -1, -2, 0, \dots 0)$, and $(1, -1, 1, -1, 0, \dots, 0)$ for configurations $(a)$, $(b)$, and $(c)$, respectively, yields $\alpha^TD\alpha >0$.  
\end{proof}
Now suppose that we can isometrically embed $H$ in a Hilbert space. Then by the previous  claim, $H$ must be a cycle, path, or complete graph. Next, we show that $H$ cannot be a cycle of length greater than three.
\begin{claim}
If $\min\limits_{\substack{\alpha \in \R^n \\ \alpha \perp \mathbbm{1}}} \alpha^T D \alpha \leq 0$, then $H \neq C_n$ for $n \geq 4$. 
\end{claim}

\begin{proof}
Suppose that $H = C_n$ for some $n \geq 4$. If $n = 2k$ is even, then by labeling the vertices sequentially (see Figure \ref{cycles}) and ordering them starting with $0, 2k-1, k-1, k$, $D$ has the following principal submatrix
$$
\left[ \begin{array}{cccc}
0 & 1 & (k-1)^2 & k^2 \\
1 & 0 & k^2 & (k-1)^2 \\
(k-1)^2 & k^2 & 0 & 1 \\
k^2 & (k-1)^2 & 1 & 0
\end{array}\right].
$$
Then setting $\alpha = (1, -1, -1, 1, 0, \dots, 0)$, $\alpha^TD\alpha = 8(k-1) > 0$. If $n = 2k+1$ is odd, then labeling the vertices sequentially (see Figure \ref{cycles}) and ordering them starting with $0, 1, k, k+1$, $D$ has the following principal submatrix
$$
\left[ \begin{array}{cccc}
0 & 1 & k^2 & k^2 \\
1 & 0 & (k-1)^2 & k^2 \\
k^2 & (k-1)^2 & 0 & 1 \\
k^2 & k^2 & 1 & 0
\end{array}\right].
$$
Then setting $\alpha = (1, -1, 1, -1, 0, \dots, 0)$, $\alpha^TD\alpha = 2(2k-3) >0$. 

\begin{figure}
\centering
\hfill
\hfill
\subfloat[Even Cycle]{
\begin{picture}(60,95) 
\put(5,35){\circle*{6}}
\put(5,85){\circle*{6}}
\put(55,35){\circle*{6}}
\put(55,85){\circle*{6}}
\put(5,85){\line(1,0){50}}
\multiput(5,35)(0,5){11}{\circle*{1}}
\put(5,35){\line(1,0){50}}
\multiput(55,35)(0,5){11}{\circle*{1}}
\put(2,90){$k$}
\put(42,90){$k-1$}
\put(2,22){$0$}
\put(50,22){$2k$}
\end{picture}
}
\hfill
\subfloat[Odd Cycle]{
\begin{picture}(60,95)
\put(5,35){\circle*{6}}
\put(5,85){\circle*{6}}
\put(55,35){\circle*{6}}
\put(55,85){\circle*{6}}
\put(30,23){\circle*{6}}
\put(5,85){\line(1,0){50}}
\multiput(5,35)(0,5){11}{\circle*{1}}
\multiput(55,35)(0,5){11}{\circle*{1}}
\put(5,35){\line(2,-1){25}}
\put(55,35){\line(-2,-1){25}}
\put(2,90){$k$}
\put(42,90){$k+1$}
\put(2,22){$0$}
\put(50,22){$2k$}
\put(15,12){$2k+1$}
\end{picture}
}
\hfill
\hfill
\hfill
\caption{Cycle Labellings}\label{cycles}      
\end{figure}
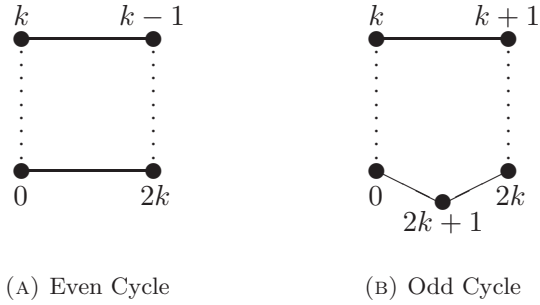
\end{proof}

Finally, we show that if $H$ is a path or complete graph then $D$ is
conditionally negative definite. First suppose that $H = K_n$ and
consider $\alpha \in \R^n$ such that $\alpha \perp \mathbbm{1}$. Then
$D = A = J-I$, $\alpha = (\alpha_1, \alpha_2, \dots, \alpha_{n-1},
-\sum_{i = 1}^{n-1} \alpha_i)$, and $\alpha^TD \alpha = -\langle
\alpha, \alpha \rangle \leq 0$. Now suppose that $H = P_n$. By our
previous observation, it suffices to show that $M := \left(I-\frac{1}{n}J\right)D\left(I-\frac{1}{n}J\right)$ is negative semidefinite. Now, $D_{ij} = (i-j)^2$ so
\begin{align*}
M_{ij} &= (i-j)^2 - \frac{1}{n}\sum_{k = 1}^n(k-j)^2-\frac{1}{n}\sum_{\ell = 1}{n}(\ell-i)^2+\frac{1}{n^2}\sum_{k = 1}^n\sum_{\ell = 1}^n(k-\ell)^2 \\
&= (i^2-2ij+j^2) -\frac{1}{n}\left( n(i^2+j^2)-2\sum_{k=1}^n k(i+j)+2\sum_{k = 1}^n k^2\right) + \frac{1}{n^2} \sum_{k =1}^n \sum_{\ell = 1}^n (k^2-2k \ell+\ell^2) \\
&= -2ij+\frac{2}{n}\binom{n+1}{2}(i+j)-\frac{2}{n^2}\binom{n+1}{2}\binom{n+1}{2} \\
&= -2\left(i-\frac{1}{n}\binom{n+1}{2}\right)\left(j-\frac{1}{n}\binom{n+1}{2} \right).
\end{align*}
Thus, $M = -2vv^T$ where $v_{i} = \left(i - \frac{1}{n} \binom{n+1}{2} \right)$ so $M$ is negative semidefinite as desired. 
\end{proof}

\subsection{Weighted Graphs}
The goal of this section is to describe the structure of isometrically
embeddable metric spaces when the associated critical graphs are
allowed to have edge weights other than zero.  Specifically, we will
provide a complete characterization of all 2-connected critical graphs
that have a weighting which admits an isometric embedding of the
associated metric space.  By simple enumeration it is easy to see that
all connected graphs on at most three vertices admit a weighting so
that the associated metric may be isometrically embedding in a Hilbert
space.  Thus, to begin our characterization we provide a complete
characterization of critical graphs on four vertices that admit an isometrically embeddable weighting.

\begin{lem}\label{L:4verts}
Let $(X,d)$ be a metric space with 4 elements that can be
isometrically embedded into a Hilbert space.  Let $G = (X,E,w)$ be the
critical graph associated with $(X,d)$, then $G$ is either a path,
$K_4$, or $K_4 - e$.
\end{lem}


\begin{proof}
Denote the elements of $X$ by $x_1, x_2, x_3, x_4$ and suppose that
$G$ is not complete. As $G$ is necessarily connected, we may assume,
without loss of generality, that $x_2 \sim x_3$, $x_3 \sim x_4$, and $x_2 \not\sim x_4$.  Further, we may also assume that  $d(x_2,x_4) = d(x_2,x_3) + d(x_3,x_4)$.   For
simplicity of notation, denote $d(x_i,x_j)$ by $d_{ij}$.  Now by
Schoenberg's Embedding Theorem, for any choice of
$\alpha_1,\alpha_2,\alpha_3 \in \R$, the vector $\alpha =
(\alpha_1,\alpha_2,\alpha_3,-\alpha_1-\alpha_2-\alpha_3)^T$ must satisfy
$\alpha^T D \alpha \leq 0$ where $D_{ij} = d_{ij}^2$.  We now note
that
\begin{align*}
\alpha^T D \alpha &= 2 \left[- \sum_{i=1}^3 d_{i4}^2 \alpha_i^2 +
                    \sum_{1 \leq i < j \leq 3} \paren{d_{ij}^2 -
                    d_{i4}^2 - d_{j4}^2}\alpha_i\alpha_j  \right] \\
&= 2 \left[- \paren{\sum_{i=1}^3 d_{i4}\alpha_i}^2+
                    \sum_{1 \leq i < j \leq 3} \paren{d_{ij}^2 -
                    \paren{d_{i4} - d_{j4}}^2}\alpha_i\alpha_j
  \right] \\
&= 2 \left[- \paren{\sum_{i=1}^3 d_{i4}\alpha_i}^2+
                   \paren{d_{12}^2 -
                    \paren{d_{14} - d_{24}}^2}\alpha_1\alpha_2 +
                   \paren{d_{13}^2 -
                    \paren{d_{14} - d_{34}}^2}\alpha_1\alpha_3\right],
\end{align*}
where the last equality comes from the fact that $d_{24} = d_{23} + d_{34}$.
Thus, if $\alpha_1$ is chosen to be $\nicefrac{-d_{24}\alpha_2
  -d_{34}\alpha_3}{d_{14}}$, we have that 
\[ \alpha^TD\alpha = 2\alpha_1\paren{\paren{d_{12}^2 -
                    \paren{d_{14} - d_{24}}^2}\alpha_2 +
                   \paren{d_{13}^2 -
                    \paren{d_{14} - d_{34}}^2}\alpha_3 }\]
Letting 
\[ u = \left[ \begin{matrix} d_{24} \\ d_{34} \end{matrix}\right] \textrm{,}
\quad v = \left[\begin{matrix} d_{12}^2 - \paren{d_{14} -
        d_{24}}^2 \\ d_{13}^2 - \paren{d_{14} -
        d_{34}}^2 \end{matrix} \right] \textrm{,} \quad\textrm{and }w =
  \left[ \begin{matrix} \alpha_2  \\ \alpha_3 \end{matrix} \right] \]
we have that
\[ \alpha^T D \alpha = - \frac{2}{d_{14}} \inner{u}{w}\inner{v}{w}.\]
But as the choice of $w$ is arbitrary in $\R^2$, this implies that
there exists some non-negative constant $c$ such that $v = cu$ as
otherwise there is a some $w$ such that $\inner{u}{w}\inner{v}{w} <
0$.

We now consider the shortest path between $x_1$ and $x_4$ in the
critical graph $G$ of $(X,d)$. This path is either $(x_1, x_2, x_3, x_4)$,
$(x_1, x_3, x_4)$, or $(x_1,x_4)$.  In the first case $G$ is a path, so consider the case where the shortest path is
$x_1,x_3,x_4$.  In this case we have that $d_{14} = d_{13} + d_{34}$, and $d_{13} < d_{12}  + d_{23}$.
Substituting the value for $d_{14}$ into the equation for $v$, we have 
\[ v =\left[\begin{matrix} d_{12}^2 - \paren{d_{13}+d_{34} -
        d_{24}}^2 \\ d_{13}^2 - \paren{d_{13} + d_{34} -
        d_{34}}^2 \end{matrix} \right]  = \left[\begin{matrix} d_{12}^2 - \paren{d_{13} -
        d_{23}}^2 \\ 0\end{matrix} \right].\]
But, as $d_{13} < d_{12} + d_{23}$ and $d_{23} < d_{13} + d_{12}$, we
have that the first component of $v$ is non-zero.  Thus, as both
components of $u$ are strictly positive, there is no
non-negative constant $c$ such that $v = cu$, a contradiction.

At this point we may assume that the critical graph contains the edge
$\set{x_1, x_4}$.  But then we have that $d_{14} < d_{13} + d_{34}$ and
$d_{34} < d_{14}  + d_{13}$ and so the second component of $v$ is
strictly positive.  By the same argument as above, this implies that
the first component of $v$ is strictly positive, and in particular
that $d_{12} > \abs{d_{14} - d_{24}}$.  Thus the required constant $c$
is strictly positive.  Now the existence of such a constant is
equivalent to $v_2 u_1 - v_1 u_2 = 0.$  Combining this with the fact
that $d_{24} =
d_{23} + d_{34}$, we get that
\begin{equation}\label{E:roots}
0 = v_2 u_1 - v_1 u_2  = d_{23}d_{13}^2 - d_{23}d_{14}^2 +
  d_{34}d_{13}^2 - d_{34}d_{12}^2 + d_{34}d_{23}^2 +
  d_{23}d_{34}^2.
\end{equation}  

At this point we may consider the shortest path between $x_1$ and
$x_2$ in the critical graph of $(X,d)$. As above, there are two
non-path choices, namely, $x_1,x_2$ and $x_1,x_3,x_2$.  First assume that the
shortest path is $x_1,x_3, x_2$ and hence $d_{12} = d_{13} + d_{23}$.
Substituting into Equation \ref{E:roots}, we have that
\begin{align*}
0 &= d_{23}d_{13}^2 - d_{23}d_{14}^2 +
  d_{34}d_{13}^2 - d_{34}\paren{d_{13}+d_{23}}^2 + d_{34}d_{23}^2 +
  d_{23}d_{34}^2 \\
&=d_{23}d_{13}^2 - d_{23}d_{14}^2 
  - 2d_{34}d_{13}d_{23}+
  d_{23}d_{34}^2 \\
&=d_{23}\left[ d_{13}^2 - d_{14}^2 
  - 2d_{34}d_{13}+
  d_{34}^2 \right]\\
&=d_{23}\left[ \paren{d_{13}-d_{34}}^2 - d_{14}^2\right].\\
\end{align*}
But this implies that $\abs{d_{13} - d_{34}} = d_{14}$ and hence,
either $d_{13} = d_{14} + d_{34}$ and $x_1 \not\sim x_{3}$ or $d_{34}
= d_{13} + d_{14}$ and $x_3 \not\sim x_4$. Both cases contradict our assumptions so we may assume that $x_1 \sim x_2$ in the critical
graph.  

We now consider whether $x_1$ and $x_3$ are adjacent in the critical
graph.  If $x_1 \not\sim x_3$, then either $d_{13} = d_{12} + d_{23}$
or $d_{13} = d_{14} + d_{34}$.  Substituting into equation
\ref{E:roots} we get 
\[ 0 = d_{23}\left[ \paren{d_{12} + d_{23} + d_{34}}^2 - d_{14}^2\right]\]
and 
\[ 0 = d_{34}\left[ \paren{d_{34} + d_{14}}^2 + d_{23}^2 +
    2d_{14}d_{34} - d_{14}^2\right], \] respectively. 
Both of these equations are obviously false, yielding
$x_1 \sim x_3$.  Thus, we have  that the critical
graph of $(X,d)$ is either a path, $K_4$, or $K_4 -e$.
\end{proof}

It is relatively easy to construct infinite families of distinct metrics with integer distances
whose critical graphs are $P_4,$ $K_4$, or $K_4-e$ and are embeddable
into a Hilbert space.  For instance, for $K_4 -e$ we have:

\begin{obs}
There exists an infinite family of isometrically embeddable metric spaces $(X,d)$ on four elements
such that the critical graph for $(X,d)$ is $K_4 - e$ and $d$ is
integer valued. Further, each metric can be isometrically embedded in $\R^2$.
\end{obs}

\begin{proof} 
Fix $z \in \N$ and $(p_i,q_i)$ for $i = 1,2,3$ such that $p_iq_i =
z$ and $p_i \equiv q_i \pmod{2}$. Define  $x_i =
\left(\frac{p_i^2-q_i^2}{2},0\right)$ for $i = 1, 2, 3$, and $x_4 =
(0,z)$. Then, since $\paren{p_iq_i, \frac{p_i^2-q_i^2}{2}, \frac{p_i^2+q_i^2}{2}}$ are three integers that form a Pythagorean triple for $i = 1,2,3$, the submetric space of $\R^2$ on these four points has $K_4-e$ as a critical graph and integer weights as desired.
\end{proof}

We can use the complete characterization of 4-element embeddable metrics
to show that paths are  the unique class of 1-connected,
weighted. isometrically embeddable, critical graphs.

\begin{lem}\label{L:trees}
Let $(X,d)$ be a metric space that can be isometrically embedded into
a Hilbert space and let $G = (X,E,d)$ be the critical graph associated with $(X,d)$. $G$ is either a path or $2$-connected and further, if $\set{u,v}$ is
a vertex cut, then $u \sim v$.
\end{lem}

\begin{proof} 
Suppose that $G$ contains a cut vertex. If every cut vertex has degree 2, $G$ is a path. Thus, suppose that $G$ contains a cut vertex $v$ of degree at least 3. Let $x,y,$ and $z$ be neighbors of $v$ such $z$ belongs to a different component of $G\setminus v$ than both $x$ and $y$. Then $x,v,z$ and $y,v,z$ are shortest paths in $G$. Thus, the critical graph of the submetric space on $\{v,x,y,z\}$ is either the claw or the claw with a single additional edge, both of which are forbidden. Thus, $G$ is either a path or 2-connected. 

Now suppose that $G$ is strictly 2-connected and let $\{u,v\}$ be a
cut set such that $u \not\sim v$. Let $C_1, C_2,\dots, C_k$ denote the
components of $G \setminus \{u,v\}$ and let $H_i := G[\set{u,v} \cup
V(C_i)]$.  Note that since $G$ has no cut vertices each $H_i$ has a
shortest path from $u$ to $v$, say $P_i$. One of these shortest paths
must be the shortest path in $G$, say $P_1$ in $H_1$.  Since $u \not\sim v$, $P_1$ and $P_2$ contain vertices $x$ and $y$, respectively, different from $u$ and $v$.

Now consider the critical graph $G'$ of the submetric on
$\{u,v,x,y\}$. Since $P_1$ is a shortest path in $G$, $d(u,x)+d(x,v) =
d(u,v)$ and so the edge $\{u,v\}$ is not present in $G'$. Further,
since no shortest path from $u$ or $v$ to $x$ in $G$ contains $y$,
$u,x,v$ is a path in $G'$. Next, observe that a shortest path from $x$
to $y$ in $G$ must go through either $u$ or $v$. Thus, the edge
$\{x,y\}$ is not present in $G'$. Now since $G'$ is connected, $y$ is
adjacent to at least one of $u$ and $v$. Without loss of generality,
assume that the edge $\{y,u\}$ is present. If the edge $\{y,v\}$ is
also present, we are done since the critical graph of $G'$ is $C_4$, a contradiction. 

Thus, suppose that $\{y,v\}$ is not present. Then $d(y,v) = d(y,u)+d(u,v)$ so $P_2$ is not a shortest path in $G$. Further, there exists a vertex $y'$ adjacent to $y$ on $P_2$ between $y$ and $v$. Now starting from $y$ and moving along $P_2$, we eventually find adjacent vertices $z_u$ and $z_v$ such that a shortest path from $z_u$ to $v$ goes through $u$ and a shortest path from $z_v$ to $v$ does not. Now consider the critical graph $G''$ of the submetric on $\{u,v,z_u,z_v\}$. Since $P_2$ is not a shortest path in $G$, the edge $\{u,v\}$ must be present. Additionally, the edge $\{z_u,z_v\}$ is present as $z_u$ and $z_v$ are adjacent in $G$. Then since a shortest path from $z_v$ to $u$ in $G$ must go through either $z_u$ or $v$, we see that the edge $\{z_v,u\}$ is not present. In the same vain, we also have that the edge $\{z_u, v\}$ is not present. Finally, by the selection of $z_u$ and $z_v$, the edges $\{z_u,u\}$ and $\{z_v,v\}$ are present. Thus, we find that $G''$ is $C_4$, a contradiction. 
\end{proof}

Finally we complete the characterization of isometrically embeddable
2-connected critical graphs by providing a complete listing of 2-connected critical graphs that have an isometric embedding.

\begin{thm}\label{T:3conn}
Let $(X,d)$ be a metric space that can be isometrically embedded into
a Hilbert space.  Let $G = (X,E,w)$ be the critical graph associated
with $(X,d)$.  If $G$ is 2-connected, but not 3-connected, then $G$
contains a Hamiltonian path $v_1, \ldots, v_{\size{X}}$ and further,
there is some $2 \leq k \leq \size{X}-1$ such that $v_i \sim v_j$ if
and only if $\abs{i-j} = 1$ or $i < k < j$.
\end{thm}

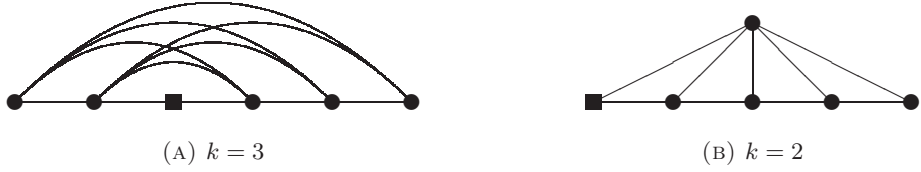
\begin{figure}
\centering
\hfill
\subfloat[$k=3$\label{F:S63}]{
\begin{picture}(150,95)
\put(0,10){\circle*{6}}
\put(30,10){\circle*{6}}
\put(90,10){\circle*{6}}
\put(120,10){\circle*{6}}
\put(150,10){\circle*{6}}
\put(57,7){\rule{6pt}{6pt}}
\qbezier(0,10)(45,55)(90,10) 
\qbezier(0,10)(60,70)(120,10) 
\qbezier(0,10)(75,85)(150,10)
\qbezier(30,10)(60,40)(90,10) 
\qbezier(30,10)(75,55)(120,10) 
\qbezier(30,10)(90,70)(150,10)
\put(0,10){\line(1,0){150}}
\end{picture}
}
\hfill
\subfloat[$k=2$\label{F:S62}]{
\begin{picture}(120,95)
\put(-3,7){\rule{6pt}{6pt}}
\put(30,10){\circle*{6}}
\put(60,10){\circle*{6}}
\put(90,10){\circle*{6}}
\put(120,10){\circle*{6}}
\put(60,40){\circle*{6}}
\put(0,10){\line(1,0){120}}
\put(60,40){\line(-1,-1){30}}
\put(60,40){\line(1,-1){30}}
\put(60,40){\line(-2,-1){60}}
\put(60,40){\line(2,-1){60}}
\put(60,40){\line(0,-1){30}}
\end{picture}
}
\hfill
\hfill
      \caption{$6$ Vertex Critical Graphs with Isometric Embeddings by Theorem \ref{T:3conn}}\label{F:crit_graphs}
\end{figure}

\begin{proof}
Since by assumption $G$ is 2-connected and not 3-connected, there
exists some pair of elements $\set{u,v}$ such that $G -\set{u,v}$ is
disconnected.  Furthermore, from Lemma \ref{L:trees} we know that $u
\sim v$.  We first show that $G - \set{u,v}$ has precisely two
components.  To that end, suppose not and let $x, y, z$ be in distinct
components.  Now consider the shortest paths between these elements,
pairwise.  Since $\set{u,v}$ is a cut and the vertices are in separate
components, each such shortest path must pass through either $u$ or
$v$.  By the pigeonhole principle, one of $u$ or $v$ has at least two
shortest paths passing through it.  Without loss of generality suppose
this is $u$ and the two shortest paths passing through it are beween
$x$ and $y$ as well as between $x$ and $z$.  Now consider the
submetric induced by $\set{x,y,z,u}$.  Since $x,y,z$ are all in
different components, vertex $u$ has degree three in the critical graph
associated to this metric.  Furthermore, since the shortest path
between $x$ and either $y$ or $z$ passes through $u$, the edges
$\set{x,y}$ and ${x,z}$ are not present in the critical graph.  Thus, the associated
critical graph is either a claw or a claw plus an edge, neither of
which are isometrically embeddable by Lemma \ref{L:4verts}.  Hence, $G
- \set{u,v}$ has precisely two components.  

Now for any $x \in X$ other than $u$ or $v$, define $V_x$ as the set
of elements $w$  in $X$ such that a shortest path between $x$ and $w$
passes through $v$ and that $w$ and $x$ are in different components of
$G - \set{u,v}$.  Define $U_x$ analogously.  We note that if $a,b \in
V_x$, then the critical graph associated to the submetric induced by
$\set{v,x,a,b}$ has no edge from $x$ to either $a$ or $b$.  Thus,
since the entire metric is isometrically embeddable, the associated
critical graph must be a path.  Furthermore, this implies that the
elements of $V_x$ are totally ordered by distance to $x$, $G[V_x]$ is
a path, and for any two vertices in $V_x$ there exists a shortest
path between them that  is entirely contained in
$G[V_x]$.  

Now suppose there exists some $a \in V_x \cap U_x$ and consider the
induced metric formed by $\set{x, u, v, a}$.  Note that we may rescale
the distances in the submetric so that $d(x,a) = 1$ without affecting
the isometric embeddablility of the submetric.  In this case, letting
$d(x,u) = \epsilon$, $d(x,v) = \delta$, and $d(u,v) = \gamma$ we have
that 
\[ D = \left[ \begin{matrix} 0 & 1 & \delta^2 & \epsilon^2 \\ 1 & 0 &
      (1-\delta)^2 & (1-\epsilon)^2 \\ \delta^2 & (1-\delta)^2 & 0 &
      \gamma^2 \\ \epsilon^2 & (1-\epsilon)^2 & \gamma^2 & 0
      \\ \end{matrix} \right] \]
up to rearrangment of the columns.  If we consider $\alpha$ of the
form $[1, -\zeta, -1, \zeta]$
then we have that 
\begin{align*}
\alpha^T D \alpha &=  -2(1+\gamma^2)\zeta  + 2\paren{\zeta\epsilon^2 +
                    \zeta(1-\delta)^2 -\delta^2 -
                    \zeta^2(1-\epsilon)^2} \\ 
&= -2 \paren{ \delta^2 - \paren{1 + \gamma^2 - \epsilon^2 -
  (1-\delta)^2}\zeta + (1-\epsilon)^2\zeta^2}.
\end{align*}
Viewing $\alpha^TD\alpha$ as a function of $\zeta$ we have that the
discriminant is 
\[
\paren{1 + \gamma^2 - \epsilon^2 - (1-\delta)^2}^2 
       -4\delta^2(1-\epsilon)^2,
\]
which is strictly positive so long as $1 + \gamma^2 - \epsilon^2 -
(1-\delta)^2 - 2\delta(1-\epsilon) > 0$.
Observing that by the triangle inequality, $\gamma^2 > \paren{\delta -
  \epsilon}^2$, we have that 
\begin{align*}
1 + \gamma^2 - \epsilon^2 - (1-\delta)^2 - 2 \delta (1-\epsilon) &>
1 + \paren{\delta -\epsilon}^2 - \epsilon^2 - (1-\delta)^2 - 2\delta(1-\epsilon) \\
&= 1+ \delta^2 -2\delta\epsilon + \epsilon^2  - \epsilon^2 - 1
  +2\delta - \delta^2 - 2\delta + 2\delta\epsilon \\
&= 0.
\end{align*}
Thus the discriminate is strictly positive and so there exists some
choice of $\zeta$ such that $\alpha^T D \alpha > 0$, contradicting the
embedability of $(X,d)$.  As a consequence we have that $V_x \cap U_x =
\varnothing$ for all $x \in X - \set{u,v}$. In particular, if $a$ and
$b$ are in different components of $G - \set{u,v}$ then there is
precisely one of $\set{u,v}$ that lies on all of the shortest paths
between $a$ and $b$.

Suppose there exists some pair $a \in U_x \cup \set{u}$ and $b \in V_x
\cup \set{v}$ such that $a \not\sim b$.  Let $w$ be a vertex on a
shortest path between $a$ and $b$ and consider the submetric
determined by $\set{a,b,w,x}$.  Without loss of generality assume that
$w$ lies in $U_x \cup \set{u}$.  Since the distances to $u$ in $U_x
\cup \set{u}$ can be totally ordered, this implies that either $w$ is
closer to $u$ than $a$ or $w$ is farther from $u$ than $a$.  In the
former case, we have that the critical graph is a claw plus an edge,
and is not isometrically embeddable in a Hilbert space.  In the later
case, recalling that $U_x \cap V_x = \varnothing$ we have that the
critical graph is $C_4$ and hence is not isometrically embeddable in a
Hilbert space.  Thus, if $x$ is a vertex adjacent to both $u$ and $v$,
there is a path $V_x, v, x, u, U_x$ and every element in $V_x \cup
\set{v}$ is adjacent to every element in $U_x \cup \set{u}$.  As a
consequence, to complete the proof it suffices to show that there exists some 2-cut in $G$
such that there is a component containing precisely one vertex.  

To that end, suppose that $V_x = \varnothing$ and let $u = u_0, u_1, u_2,
\ldots, u_t$ be the ordering of elements of $U_x \cup \set{u}$ in terms of
increasing distance to $u$.  Now note that the only neighbors of $u_t$
are $u_{t-1}$ and $v$ and hence $\set{u_{t-1},v}$ is a 2-cut with a
component having precisely one element.  By applying an analogous
argument in the case $U_x = \varnothing$, we may assume without loss
of generality that for any vertex $x \in X -\set{u,v}$, both $U_x$ and $V_x$ are
non-empty.  

Now choose some $a \in U_x$ and $b \in V_x$ and consider $V_a$.  Since
$U_a$ and $V_a$ are disjoint, $x \not\in V_a$, and thus there exists
some $z \in V_a$ such that $z \neq x$.  Now suppose that the shortest
paths between $b$ and $z$ goes through $v$, and consider the submetric
induced by $\set{v,x,b,z}$. As all shortest paths between $x$ and $b$,
as well as $z$ and $b$ go through $v$, the associated critical graph
has $v$ as vertex of degree $3$ and no edges between $\set{x,z}$ and
$b$.  Thus the associated critical graph is either a claw or a claw
plus an edge, and hence is not isometrically embeddable in a Hilbert
space.  Hence, we have that the shortest paths between $b$ and $z$
pass through $u$.  Now, by the fact that all shortest paths between
vertices in opposite components pass through precisely one of $u$ and
$v$, we have that 
\begin{align*}
d(x,u) + d(u,a) &< d(x,v) + d(v,a) \\
d(x,v) + d(v,b) &< d(x,u) + d(u,b) \\
d(z,u) + d(u,b) &< d(z,v) + d(v,b) \\
d(z,v) + d(v,a) &< d(z,u) + d(u,a).
\end{align*}
Summing these inequalities and cancelling we get $0 < 0$, a
contradiction.  Thus there exists some vertex $x$ such that one of
$U_x$ and $V_x$ is empty, completing the proof.  
\end{proof}

It is easy to find weighting for the classe graphs described in Theorem
\ref{T:3conn} so that it is the critical graph of a metric space
which is isometrically embeddable in a Hilbert space.  Specifically, 
fix some pair of positive integers, $k$ and $n$ such that $2 \leq k
\leq n$ and let $S_{n,k} = \set{(0,0)} \cup \set{(i,0) \colon 1 \leq i < k} \cup
\set{(0,j) \colon 1 \leq n-k} \subseteq \N^2$.  Consider the critical
graph induced by $(S,d)$ where $d$ is the standard distance metric on
$\N^2$.  The critical graph $G$ for this metric consists of a path $(k-1,0), (k-2, 0),
\ldots, (0,0), (0,1), (0,2), \ldots, (0,n-k)$ as well as edges between
any pair of points $(x,0)$ and $(0,y)$ where $x,y \neq 0$.  That is,
$G$ has a Hamiltonian path with an identified non-endpoint vertex,
$(0,0)$, and every pair of vertices on opposite sides of $(0,0)$ are
connected by an edge.  Observing that $\N^2$ sits isometrically in
the standard Hilbert space on $\R^2$ completes the construction.

\begin{figure}
\centering
\hfill
\subfloat[Embedding Yielding Figure \ref{F:S63}]{
\begin{picture}(150,90)
\multiput(15,15)(5,0){25}{\circle*{1}}
\multiput(15,45)(5,0){25}{\circle*{1}} 
\multiput(15,75)(5,0){25}{\circle*{1}} 
\multiput(30,0)(0,5){19}{\circle*{1}}
\multiput(60,0)(0,5){19}{\circle*{1}}
\multiput(90,0)(0,5){19}{\circle*{1}}
\multiput(120,0)(0,5){19}{\circle*{1}}
\put(27,12){\rule{6pt}{6pt}}
\put(30,45){\circle*{6}}
\put(30,75){\circle*{6}}
\put(60,15){\circle*{6}}
\put(90,15){\circle*{6}}
\put(120,15){\circle*{6}}
\end{picture}
}
\hfill
\subfloat[Embedding Yielding Figure \ref{F:S62}]{
\begin{picture}(150,90) 
\multiput(0,15)(5,0){31}{\circle*{1}}
\multiput(0,45)(5,0){31}{\circle*{1}} 
\multiput(15,0)(0,5){13}{\circle*{1}}
\multiput(45,0)(0,5){13}{\circle*{1}}
\multiput(75,0)(0,5){13}{\circle*{1}}
\multiput(105,0)(0,5){13}{\circle*{1}}
\multiput(135,0)(0,5){13}{\circle*{1}}
\put(12,12){\rule{6pt}{6pt}}
\put(15,45){\circle*{6}}
\put(45,15){\circle*{6}}
\put(75,15){\circle*{6}}
\put(105,15){\circle*{6}}
\put(135,15){\circle*{6}}
\end{picture}
}
\hfill
\hfill
\caption{Embeddings in $\R^2$ of Metrics with Critical Graphs 
        Given in Figure \ref{F:crit_graphs}}\label{F:embedding}
\end{figure}
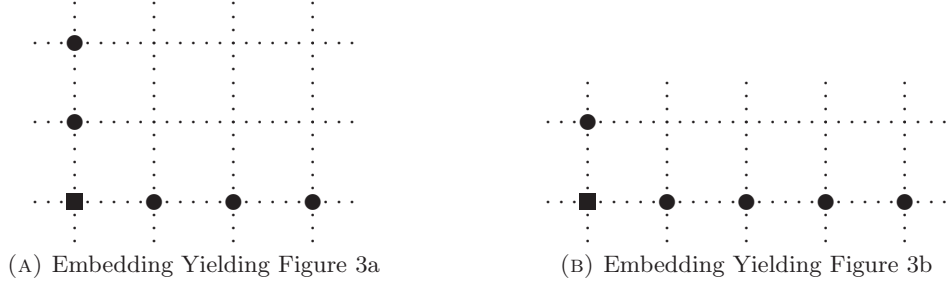
\section{Geometric Spectrum of the Normalized
  Laplacian} \label{S:GeomSpec}
In this section, we will show how the ability to embed a metric space
$(X,d)$ isometrically allows the concept of a geometric Fiedler vector
to be extended to a geometric spectra  for the graph with
respect to the metric space $(X,d)$.  For the sake of specificity, we
will use the geometric Fiedler vector for the normalized Laplacian as
defined by Radcliffe and Williamson~\cite{Radcliffe:GeometricEigs} and Dumitriu and
Radcliffe~\cite{Radcliffe:GeometricExpansion}.  To that end, recall that given a connected graph $G = (V,E)$, the \emph{normalized Laplacian} is
defined as 
\[
\mathcal{L} = I-T^{-1/2}AT^{-1/2}
\]
where $A$ and $T$ are the adjacency and diagonal degree matrix of $G$, respectively. Specifically,
\[
\mathcal{L}_{xy} =
\begin{cases}
1 & \text{if } x = y \\
\frac{-1}{\sqrt{d_xd_y}} & \text{if }\{x,y\} \in E \\
0 & \text{otherwise}
\end{cases},
\]
where $d_x$ is the degree of vertex $x$ in the graph $G$.  Now the $0$
eigenvalue for $\mathcal{L}$ is given by $T^{\half}\one$ and thus by
the Courant-Fisher Theorem,
\[ \lambda_2 = \min_{x \perp T^{\half} \one}
  \frac{\inner{x}{\mathcal{L}x}}{\inner{x}{x}}. \]

Following the standard transformation in~\cite{Chung:spectral}, we
have that
\[ 
  \frac{\inner{x}{\mathcal{L}x}}{\inner{x}{x}} =
  \frac{\inner{x}{T^{-\half}LT^{-\half}x}}{\inner{x}{x}} =
  \frac{\inner{T^{-\half}x}{LT^{-\half}{x}}}{\inner{x}{x}} =
  \frac{\inner{f}{Lf}}{\inner{T^{\half}f}{T^{\half}f}} =
  \frac{\sum_{u \sim
      v}(f(u)-f(v))^2}{\sum_{v}f(v)^2d_v}, \]
where $f = T^{-\half}x$.  We will follow the terminology of 
\cite{Chung:spectral} and refer to $f$ as a \emph{harmonic
  eigenvector} for $\mathcal{L}$.  It is worth noting that the harmonic
eigenvectors generate the same spectrum as the standard
eigenvectors.  In particular,
\[ \lambda_2 =\min_{f \perp T\one} \frac{\sum_{u \sim
      v}(f(u)-f(v))^2}{\sum_{v}f(v)^2d_v}.\]
Furthermore, following the derivation in \cite{Chung:spectral} the
restriction on $f$ can be relaxed;
\begin{align*}
\lambda_2 = \min_{f \perp T\one} \frac{\sum_{u \sim
      v}(f(u)-f(v))^2}{\sum_{v}f(v)^2d_v} &= \min_f \frac{\sum_{u \sim
      v}(f(u)-f(v))^2}{\sum_{v}\paren{f(v) - \bar{f}}^2d_v} \\ &= \min
  _{f \textrm{ non-constant}} \frac{\vol(G) \sum_{u \sim
      v}(f(u)-f(v))^2}{\sum_{u,v}\paren{f(u)-f(v)}^2d_ud_v},
\end{align*}
where $\bar{f} = \frac{\sum_v f(v) d_v}{\vol(G)}$ and $\vol(G) = \sum_v
d_v$. From this formulation, the geometric Fiedler vector can be
defined by noticing that $\paren{f(u)-f(v)}^2$ can be viewed as the
squared distance between $f(u)$ and $f(v)$ on $\R$.  Thus the
geometric Fiedler value of the normalized Laplacian over the metric
$(X,d)$ is defined as
\[ \lambda_2(G,X) = \min_{\substack{f \colon V \rightarrow X \\ f \text{ non-constant}}} \frac{\vol(G)\sum_{u \sim v}d(f(u),f(v))^2}{\sum_{u,v}d(f(u),f(v))^2d_ud_v}.\]

Now suppose that $G$ can be isometrically embedded via $\phi$ into a
Hilbert space $\mathcal{H}$ with kernel function $K$.  Denote by
$\norm[\mathcal{H}]{\cdot}$ and $\inner[\mathcal{H}]{\cdot}{\cdot}$
the norm and inner product on $\mathcal{H}$.  For the sake of
simplicity we will assume that $\mathcal{H} = \SPAN\set{\phi(X)}$.  We
then have
that 
\[ \lambda_2(G,X) = \min_{\substack{f \in \phi(X)^{\size{V}} \\ f
      \text{ non-constant}}} \frac{\vol(G)\sum_{u \sim
      v}\norm[\mathcal{H}]{f(u)-f(v)}^2}{\sum_{u,v}\norm[\mathcal{H}]{f(u)-f(v)}^2d_ud_v}.\]
At this point it would be natural to define orthogonality in terms of
$\mathcal{H}$, however we recall that $f \in \mathcal{H}$ is analogous to a harmonic
eigenfunction of $\mathcal{L}$ over $\R$.  Hence, it is necessary to invert the
process of defining the harmonic eigenfunction in order to determine
the appropriate notion of orthogonality.  To that end we first need
the following multi-dimensional generalization of a standard algebraic
fact.
\begin{lem}
Let $\alpha_1,\ldots,\alpha_n \in \R^k$ and let $\alpha = \frac{1}{n}
\sum_i \alpha_i$, then  \[ n \sum_i \norm{\alpha_i - \alpha}^2 = \sum_{i
  < j} \norm{\alpha_i - \alpha_j}^2.\]
\end{lem} 
\hidden{
\begin{proof}
We note that
\begin{align*}
n \sum_i \norm{\alpha_i - \alpha}^2 &= n \sum_i \inner{\alpha_i - \alpha}{\alpha_i -
  \alpha} \\
&= n \sum_i \paren{ \inner{\alpha_i}{\alpha_i} - 2\inner{\alpha_i}{\alpha} +
  \inner{\alpha}{\alpha} }\\
&= n \sum_i  \paren{\norm{\alpha_i}^2 - \frac{2}{n}\sum_j \inner{\alpha_i}{\alpha_j} +
 \frac{1}{n^2}\sum_s \sum_t \inner{ \alpha_s}{\alpha_t} } \\
&= \sum_i  n \norm{\alpha_i}^2 - 2\sum_i\sum_j \inner{\alpha_i}{\alpha_j} +
\sum_s \sum_t \inner{ \alpha_s}{\alpha_t}  \\
&= \sum_i  n \norm{\alpha_i}^2 - \sum_i\sum_j
  \inner{\alpha_i}{\alpha_j}  \\
&= \sum_i  (n-1) \norm{\alpha_i}^2 - \sum_{i < j}
  2\inner{\alpha_i}{\alpha_j}  \\
&= \sum_{i < j} \norm{\alpha_i}^2 - 2\inner{\alpha_i}{\alpha_j} +
  \norm{\alpha_j}^2 \\
&= \sum_{i < j} \norm{\alpha_i - \alpha_j}^2,
\end{align*}
as desired.
\end{proof}
}
Thus we have that 
\[
\lambda_2(G,X) = \min_{\substack{f \in \phi(X)^{\size{V}} \\ f 
      \text{ non-constant}}} \frac{\sum_{u \sim v} \norm[\mathcal{H}]{f(u) -
      \bar{f} + \bar{f}
      -f(v)}^2}{\sum_{u}\norm[\mathcal{H}]{f(u)-\bar{f}}^2d_u},\]
where $\bar{f} = \frac{1}{\vol(G)} \sum_v f(v)d_v$. Letting $g(u)
= f(u) - \bar{f}$ and $P = \set{f - \one_{\size{V}} \otimes \bar{f} \colon f \in \phi(X)^{\size{V}}}$, this can be rewritten as 
\begin{align*}
\lambda_2(G,X) &= \min_{g \in P -\set{0}} \frac{\sum_{u\sim v} \norm[\mathcal{H}]{g(u) 
      -g(v)}^2}{\sum_{u}\norm[\mathcal{H}]{g(u)}d_u} \\
 &= \min_{g \in P - \set{0}} \frac{\sum_{u \sim v} \inner[\mathcal{H}]{g(u) -
      g(v)}{g(u) - g(v)}}{\sum_{u}\inner[\mathcal{H}]{g(u)}{g(u)}
  d_u} \\
&= \min_{g \in P - \set{0}} \frac{\sum_{u \sim v} \inner[\mathcal{H}]{g(u) -
      g(v)}{g(u) -
  g(v)}}{\sum_{u}\inner[\mathcal{H}]{\sqrt{d_u}g(u)}{\sqrt{d_u}g(u)}}
  \\
&=\min_{g \in P - \set{0}} \frac{\sum_{u \sim v} \inner[\mathcal{H}]{g(u) -
      g(v)}{g(u) -
  g(v)}
}{\inner[\mathcal{H}^{\size{V}}]{T^{\half}\otimes I_{\mathcal{H}}
  g}{T^{\half}\otimes I_{\mathcal{H}} g}} \\
&=\min_{g \in P - \set{0}} \frac{
\inner[\mathcal{H}^{\size{V}}]{L^{\half} \otimes I_{\mathcal{H}} g}{L^{\half} \otimes I_{\mathcal{H}} g}}{\inner[\mathcal{H}^{\size{V}}]{T^{\half}\otimes I_{\mathcal{H}}
  g}{T^{\half}\otimes I_{\mathcal{H}} g}} 
\end{align*}
where the last equality follows from standard transformations of the
combinatorial Laplacian combined with the fact that $L$ is positive
semi-definite and has a matrix square root.  
Now, since $(X,d)$ is a finite metric space and $\mathcal{H} =
\SPAN\set{\phi(X)}$, we may think of $g$ as an element of $\R^{k
  \size{V}}$ for some finite $k$.   Thus, we have
\[ \lambda_2(G,X) = \min_{\substack{g \in P\\ g \perp \SPAN\set{T \otimes e_i}}} \frac{
    \paren{L^{\half} \otimes I_k g}^T \paren{I_{\size{V}} \otimes
      K} \paren{L^{\half} \otimes I_k g}}{\paren{T^{\half}\otimes I_k
      g}^T \paren{I_{\size{V}} \otimes K} \paren{T^{\half}\otimes I_k
      g}}. \] Since $K$ is a kernel function, both $K$ and $L$ are
positive semi-definite and have unique matrix square roots.  And so,
letting $h = T^{\half} \otimes K^{\half} g$, we can see that the
minimization can be re-expressed as a Raleigh quotient in terms of $h$
for $\mathcal{L} \otimes I_k$. As a consequence we see that two
non-constant functions $f_1 \colon V \rightarrow X$ and $f_2 \colon V
\colon \rightarrow X$ should be considered orthogonal if and only if 
\begin{align*} 
0 &= \inner{(T^{\half} \otimes K^{\half}) \paren{f_1 -
    \overline{f_1}}}{(T^{\half} \otimes K^{\half}) \paren{f_2 -
    \overline{f_2}}} \\
&= \inner[\mathcal{H}^{\size{V}}]{(T^{\half} \otimes I_{\mathcal{H}} ) \paren{f_1 -
    \overline{f_1}}}{(T^{\half} \otimes I_{\mathcal{H}}) \paren{f_2 -
    \overline{f_2}}}   \\
&= \sum_v d_v \inner[\mathcal{H}]{f_1(v) - \overline{f_1}(v)}{f_2(v) - \overline{f_2}(v)}\\
&= \sum_v d_v \inner[\mathcal{H}]{f_1(v)}{f_2(v)}  -
  d_v\inner[\mathcal{H}]{\overline{f_1}(v)}{f_2(v)} -
  d_v\inner[\mathcal{H}]{f_1(v)}{\overline{f_2}(v)} + d_v\inner[\mathcal{H}]{\overline{f_1(v)}}{\overline{f_2}(v)}\\
&= \sum_v d_v \inner[\mathcal{H}]{f_1(v)}{f_2(v)}  - 2\sum_{v,u}\frac{d_vd_u}{\vol(G)}\inner[\mathcal{H}]{f_1(u)}{f_2(v)} +
  \sum_{v,x,y}
  \frac{d_vd_xd_y}{\vol(G)^2}\inner[\mathcal{H}]{f_1(x)}{f_2(y)}\\
&= \sum_v d_v \inner[\mathcal{H}]{f_1(v)}{f_2(v)}  - 
  \sum_{x,y}
  \frac{d_xd_y}{\vol(G)}\inner[\mathcal{H}]{f_1(x)}{f_2(y)}\\
\end{align*}

Before continuing, we should observe that if we relax the discrete
nature of the embedding, i.e. remove the restriction that $g \in P$,
we end up with the spectrum of the operator $\mathcal{L} \otimes
I_{\mathcal{H}}$ and an eigendecomposition that can be readily
recovered from the standard decomposition of $\mathcal{L}$.  Thus we
insist on the discrete nature of the embedding and note that that if $\mathcal{H}$, $K$, and $\phi$ come from
the Schoenberg construction for point $x_0$, then the orthogonality condition can be simplified
further by observing that for any two points
in $\phi(X)$ the inner product is explicitly given by the kernel.
Specifically, let $\tilde{f}_1 \colon V \rightarrow X$ and
$\tilde{f_2} \colon V \rightarrow X$ be two non-constant functions
from $V$ to the metric space.  They should be considered orthogonal if
and only if 
\begin{align*}
0 &= \sum_v \frac{d_v}{2} \paren{ d(\tilde{f}_1(v), x_0)^2 +
    d(\tilde{f}_2(v), x_0)^2 - d(\tilde{f}_1(v),\tilde{f}_2(v))^2} \\
  &\qquad -
    \sum_{u,v} \frac{d_ud_v}{2\vol(G)} \paren{ d(\tilde{f}_1(v), x_0)^2 +
    d(\tilde{f}_2(u), x_0)^2 - d(\tilde{f}_1(v),\tilde{f}_2(u))}^2 \\
&= 
  \sum_{u,v} \frac{d_ud_v}{2\vol(G)} d(\tilde{f}_1(v),\tilde{f}_2(u))^2  - \sum_v \frac{d_v}{2} d(\tilde{f}_1(v),\tilde{f}_2(v))^2.
\end{align*}
It is worth noting that this definition is independent of the choice
of base point, $x_0$, and furthermore does not depend on the
existence of a particular embedding of $(X,d)$ into $\mathcal{H}$.
Thus, it also provides a plausible definition for orthogonality of
mappings when $(X,d)$ is not isometrically embeddable in a Hilbert
space.  

Furthermore, if we consider let $(X,d)$ be the standard metric space
on $\R$ and recall that
the harmonic eigenvectors for $\mathcal{L}$ over $\R$ satisfy $f \perp T\one$, we have that 
the orthogonality condition is:
\begin{align*}
0 &= \sum_{u,v} \frac{d_ud_v}{2\vol(G)} \paren{\tilde{f}_1(v) -
    \tilde{f}_2(u)}^2  - \sum_v \frac{d_v}{2} \paren{\tilde{f}_1(v) -
    \tilde{f}_2(v)}^2 \\
&= \sum_{u,v} \frac{d_ud_v}{2\vol(G)} \paren{\tilde{f}_1(v)^2 -
  2\tilde{f}_1(v)\tilde{f}_2(u) +
    \tilde{f}_2(u)^2}  - \sum_v \frac{d_v}{2} \paren{\tilde{f}_1(v)^2 
  - 2\tilde{f}_1(v)\tilde{f}_2(v) + 
    \tilde{f}_2(v)^2} \\
&= \sum_v d_v 
  \tilde{f}_1(v)\tilde{f}_2(v) -\sum_{u,v} \frac{d_ud_v}{\vol(G)} 
  \tilde{f}_1(v)\tilde{f}_2(u)   \\
&= \sum_v d_v 
  \tilde{f}_1(v)\tilde{f}_2(v),
\end{align*}
which is precisely the condition for orthogonality for harmonic eigenvectors.

\bibliographystyle{siam}
\bibliography{references}

\end{document}